\documentclass[12pt]{amsart}
\usepackage[alphabetic]{amsrefs}
\usepackage{mathdots, blkarray}
\usepackage{multicol}
\usepackage{graphicx}
\usepackage{amscd}
\usepackage{upgreek}
\usepackage{stmaryrd}
\usepackage{longtable}
\usepackage[T1]{fontenc}
\usepackage{latexsym, amsmath, amssymb, amsthm}
\usepackage[svgnames]{xcolor}
\usepackage{rsfso}
\usepackage[mathscr]{eucal}
\usepackage{mathtools}
\usepackage{mathptmx}
\usepackage{titletoc}
\usepackage{wrapfig}
\usepackage{float}
\usepackage{xypic}
\usepackage{microtype}
\usepackage{dsfont}
\usepackage{xcolor}
\usepackage{color}
\usepackage[colorlinks = true,
            linkcolor  = DarkBlue,
            urlcolor   = DarkRed,
            citecolor  = DarkGreen]{hyperref}
\usepackage{hyperref}
\allowdisplaybreaks	
\usepackage[centering, includeheadfoot, hmargin=1.0in, tmargin=0.5in,
  bmargin=0.6in, headheight=6pt]{geometry}
\newtheorem{theorem}{Theorem}[section]

\newtheorem{lem}[theorem]{Lemma}
\newtheorem{coro}[theorem]{Corollary}

\newtheorem{thm}[theorem]{Theorem}

\newtheorem{rem}[theorem]{\rm\textsc{Remark}}
\newtheorem{exam}[theorem]{\rm\textsc{Example}}

\newcommand{\ideal}[1]{\ensuremath{\left\langle #1 \right\rangle}}

\DeclareMathOperator{\Quot}{Quot}

\DeclareMathOperator{\dia}{diag}

\DeclareMathOperator{\GL}{GL}

\DeclareMathOperator{\LT}{LT}

\newcommand{\B}{\mathcal{B}}

\newcommand{\N}{\mathbb{N}}
\newcommand{\F}{\mathbb{F}}
\newcommand{\R}{\mathcal{R}}
\newcommand{\tete}{\text{t\^ete-\`a-t\^ete} }

\newcommand{\ra}{\longrightarrow}

\newcommand{\wt}{\widetilde}
\newcommand{\ol}{\overline}

\newcommand{\hbo}{$\hfill\Diamond$}

\begin{document}
\title{Modular invariants of a vector and a covector for some elementary abelian $p$-groups}
\def\shorttitle{Modular invariants of a vector and a covector for some elementary abelian $p$-groups}

\author{Shan Ren}

\address{School of Mathematics and Statistics, Northeast Normal University, Changchun 130024, China}
\email{rens734@nenu.edu.cn}

\begin{abstract}
Let $\F_p$ be the prime field of order $p>0$ and $G$ be an elementary abelian $p$-group.
For some $n$-dimensional cohyperplane $G$-representations $V$ over $\F_p$, we show that $\F_p[V\oplus V^*]^G$, the invariant ring of a vector and a covector is a complete intersection by exhibiting an explicit generating set (in fact, a SAGBI basis) and exposing all relations among the generators.
\end{abstract}

\date{\today}
\thanks{2020 \emph{Mathematics Subject Classification}. 13A50.}
\keywords{Modular invariants; complete intersections; elementary abelian $p$-groups.}
\maketitle \baselineskip=16pt

\dottedcontents{section}[1.16cm]{}{1.8em}{5pt}
\dottedcontents{subsection}[2.00cm]{}{2.7em}{5pt}

\section{Introduction}
\setcounter{equation}{0}
\renewcommand{\theequation}
{1.\arabic{equation}}
\setcounter{theorem}{0}
\renewcommand{\thetheorem}
{1.\arabic{theorem}}

\noindent Let $k$ be a field of any characteristic $p$, $G$ be a finite group, and $W$ be a finite-dimensional faithful representation of $G$ over $k$. Let $W^*$ denote the dual space of $W$. The action of $G$ on $W^*$ can be extended algebraically to an action of $G$ on the symmetric algebra $k[W]:=S(W^*)$ on $W^*$ defined by $(g\cdot f)(w):=f(g^{-1}\cdot w)$ for all
$g\in G,f\in k[W]$, and $w\in W$. The subalgebra $k[W]^G$ of $k[W]$ consisting of all $G$-invariant elements, called the \textit{invariant ring} of $G$ on $W$, is the main object of study in algebraic invariant theory. We say that $k[W]^G$ is \textit{modular} if $p>0$ divides the order of $G$, otherwise \textit{non-modular}. Compared with the non-modular case, modular invariant rings usually have more complicated structures. See \cite{CW11} and \cite{DK15} for general references to the invariant theory of finite groups.

Computing modular invariants of $p$-groups occupies a central position in modular invariant theory of finite groups.
Lots of works have been done for the cyclic group of order $p$ with its indecomposable representations; see, for example,
\cite{CSW10, Weh13}, and references therein. It is well-known that the difficulty of computing modular invariants of $p$-groups
relies on the order of the group $G$ and the structure of the representation $W$. However, for some special
types of representations, modular invariants of finite $p$-groups with high powers could be computable.
In 2011, Bonnaf\'e and Kemper \cite{BK11} explicitly computed the modular invariant ring of \textit{a vector and a covector}:
$$\F_q[V\oplus V^*]^{U_n(\F_q)}$$
where $U_n(\F_q)$ denotes the group of all upper triangular matrices with diagonal entries all 1 and
$V$ denotes the standard representation of $U_n(\F_q)$. Recently, a few works on modular invariant rings of vectors and covectors
for finite classical groups have been appeared; see \cite{Che14, CW17, CW19, CT19}, and \cite{Che21}.

Motivated by seeking an understanding of the difference between modular invariants of two vectors (\cite{Chu03}) and modular invariants of a vector and a covector, the purpose of this article is to compute the modular invariant rings of a vector and a covector for some elementary abelian $p$-groups. We elucidate the difference between the case of two vectors and the case of one vector and one covector in Remark \ref{rem1.4} below.

Let us consider an elementary abelian $p$-group $G$ of rank $n-1$, generated by $\{g_1,g_2,\dots,g_{n-1}\}$.
Let $V$ be an $n$-dimensional representation of $G$ over the prime field $\F_p$ and $\{y_n, y_{n-1},\dots, y_1\}$ be a given basis of $V$ with respect to which the matrix of $g_i$ has the form $$I_n+a_i\cdot E_{1,i+1},$$ where $a_i\in\F_p^{\times}$ and $E_{1,i+1}$ denotes the $n\times n$-matrix in which the $(1,i+1)$-entry is 1 and other entries are zero.
Choose a basis $\{x_n,x_{n-1},\dots,x_1\}$ for $V^*$ dual to $\{y_n, y_{n-1},\dots, y_1\}$ and identify $\F_p[V\oplus V^*]$ with
the polynomial ring $\F_p[x_n,x_{n-1},\dots,x_1, y_n, y_{n-1},\dots, y_1]$. We take the grevlex ordering on $\F_p[V\oplus V^*]$
with $x_n>x_{n-1}>\cdots>x_1>y_n>y_{n-1}>\cdots>y_1$. Thus for all $i,j\in\{1,2,\dots,n-1\}$
\begin{eqnarray}
g_i\cdot x_n-x_n\in\F_p^{\times}\cdot x_{n-i}, && g_i\cdot x_j-x_j=0; \label{eq2.1}\\
g_i\cdot y_n-y_n=0, & & g_i\cdot y_j-y_{j}\in\F_p\cdot y_n. \label{eq2.2}
\end{eqnarray}

Embedding $\F_p[V]^G$ and $\F_p[V^*]^G$ into subalgebras of
$\F_p[V\oplus V^*]^G$, it follows from \cite[Proposition 16]{Kem96} and \cite[Lemma 2.6.3]{CW11} that 
$$\F_p[V]^G=\F_p[N(x_n),x_{n-1},\dots,x_{1}]\text{ and } \F_p[V^*]^G=\F_p[y_n,N(y_{n-1}),\dots, N(y_1)]$$
both are polynomial algebras, where
\begin{equation}
\label{eq2.3}
N(y_i):=y_i^p-y_n^{p-1}\cdot y_i
\end{equation}
 for all $i=1,2,\dots,n-1$, and
\begin{equation}
\label{eq2.4}
N(x_n):=x_n^{p^{n-1}}+\sum_{i=1}^{n-1} (-1)^i\cdot d_{i,n-1}\cdot x_n^{p^{n-1-i}}
\end{equation}
where $d_{1,n-1},\dots, d_{n-1,n-1}$ denote the Dickson invariants in variables $x_1,\dots,x_{n-1}$; see \cite{Dic11}
for the original reference on Dickson invariants or for example, \cite[Exercise 2]{FS16} and \cite[Section 2]{CT19} for a more modern description. 

The natural pairing of $V$ and $V^*$ gives rise to some $\GL_n(\F_p)$-invariants in $\F_p[V\oplus V^*]$:
\begin{equation}
\label{eq2.5}
u_j:=x_1^{p^j}\cdot y_1+x_2^{p^j}\cdot y_2+\dots+x_n^{p^j}\cdot y_n
\end{equation}
for $j=0,1,\ldots, n-2$. 

The main result of this paper is summarized as follows.

\begin{thm}\label{thm1}
The invariant ring $\F_p[V\oplus V^*]^G$ is a complete intersection generated by
$$\B:=\{x_1,\dots,x_{n-1},N(x_n), y_n,N(y_{n-1}),\dots, N(y_1),u_0,\dots,u_{n-2}\}.$$
Moreover, $\B$ is a SAGBI basis of $\F_p[V\oplus V^*]^G$ with respect to the
grevlex ordering: $x_n>x_{n-1}>\cdots>x_1>y_n>y_{n-1}>\cdots>y_1$.
\end{thm}

The following two low-dimensional examples illustrate this theorem.

\begin{exam}{\rm
Consider $n=2$ and $G=\ideal g$ is an elementary abelian $p$-group of rank $1$. Let $\{y_2, y_1\}$ be a basis of $V$ and $\{x_2, x_1\}$ be the basis for $V^*$ dual to $\{y_2, y_1\}$. The matrix of $g$ with respect to the basis $\{x_2, x_1, y_2, y_1\}$ of
$V\oplus V^*$ has the form
$$g=\begin{pmatrix}
1 & 0 & 0 & 0 \\
-1& 1 & 0 & 0 \\
0 & 0 & 1 & 1 \\
0 & 0 & 0 & 1
\end{pmatrix}.$$
Then $N(x_2)=x_2^p-x_1^{p-1}\cdot x_2$, $N(y_1)=y_1^p-y_2^{p-1}\cdot y_1$ and $u_0=x_1\cdot y_1+ x_2\cdot y_2$.
Theorem \ref{thm1} asserts that  $\F_p[V\oplus V^*]^G$ is a hypersurface,  generated by $\{x_1, N(x_2), y_2, N(y_1), u_0\}$, subject to
$$u_0^p-(x_1\cdot y_2)^{p-1}\cdot u_0-x_1^p\cdot N(y_1)-N(x_2)\cdot y_2^p=0,$$
the unique relation among these five generators.
\hbo}\end{exam}

\begin{exam}\label{exam103}
{\rm
Suppose that $n=3$ and $G$ is an elementary abelian $p$-group of rank $2$ with two generators $g_1$ and $ g_2$. Let $\{y_3, y_2, y_1\}$ be a basis for $V$ and $\{x_3, x_2, x_1\}$ be the basis of $V^*$ dual to $\{y_3, y_2, y_1\}$. The generators $g_1, g_2$ acting on $V\oplus V^*$ with respect to the basis $\{x_3, x_2, x_1, y_3, y_2, y_1\}$ have the matrix forms:
$$g_1=\dia\left\{\begin{pmatrix}
1 & 0 & 0 \\
-1& 1 & 0 \\
0 & 0 & 1
\end{pmatrix},\begin{pmatrix}
 1 & 1 & 0  \\
0 & 1 & 0  \\
0 & 0 & 1
\end{pmatrix}\right\},~~ g_2=\dia\left\{\begin{pmatrix}
1 & 0 & 0 \\
0& 1 & 0 \\
-1 & 0 & 1
\end{pmatrix},\begin{pmatrix}
 1 & 0& 1  \\
0 & 1 & 0  \\
0 & 0 & 1
\end{pmatrix}\right\}.$$
Note that
$$N(x_3)=x_3^{p^2}-d_{1, 2}\cdot x_3^p+ d_{2, 2}\cdot x_3,$$ where
$d_{1,2}$ and $d_{2,2}$ are the two-dimensional Dickson invariants in $x_1,x_2$ of degrees  $p^2-p$ and $p^2-1$, respectively; see, for example, \cite[Section 2]{CW17}.
Moreover,
\begin{eqnarray*}
N(y_1)=y_1^p-y_3^{p-1}\cdot y_1, && N(y_2)=y_2^p-y_3^{p-1}\cdot y_2,\\
u_0=x_1\cdot y_1+x_2\cdot y_2+x_3\cdot y_3, && u_1=x_1^p\cdot y_1+x_2^p\cdot y_2+x_3^p\cdot y_3.
\end{eqnarray*}
The invariant ring $\F_p[V\oplus V^*]^G$ is generated by $\{x_1, x_2, N(x_3), y_3, N(y_2), N(y_1),u_0,u_1\}$, subject to the following two relations:
\begin{eqnarray*}
&&u_0^p-u_1\cdot y_3^{p-1}-x_1^p\cdot N(y_1)-x_2^p\cdot N(y_2)=0,\\
&&u_1^p-N(x_3)\cdot y_3^p-d_{1, 2}\cdot u_0^p+d_{2, 2}\cdot u_0\cdot y_3^{p-1}-\sum_{i=1}^{2}(x_i^{p^2}-d_{1, 2}\cdot x_i^p)\cdot N(y_i)=0.
\end{eqnarray*}
Note that $d_{1,2}$ and $d_{2,2}$ are polynomials in $x_1,x_2$, and both $x_1,x_2$ are invariants.
\hbo}\end{exam}

\begin{rem}\label{rem1.4}{\rm
Recall in \cite{Chu03}  that for an $n$-dimensional modular representation $V$ of an elementary abelian $p$-group $G$, the two-vector invariant ring $\F_p[2V]^G$ is either a complete intersection or not a Cohen-Macaulay ring; see \cite[Corollary 4.5]{Chu03}. In particular, when $V$ is a cohyperplane representation of $G$,  $\F_p[2V]^G$ is a complete intersection if $n=2$ and not Cohen-Macaulay if $n>2$; see \cite[Example 4.6]{Chu03}. However,  \cite[Example 4.6]{Chu03} didn't
include an explicit generating set of $\F_p[2V]^G$ for a general $n>2$.

In fact, a Magma computation \cite{BCP97} indicates that finding a generating set of $\F_p[2V]^G$ ($n\geqslant 3$) should be more complicated than finding a generating set for $\F_p[V\oplus V^*]^G$.
For example, when $(n,p)=(3,2)$,  $\F_2[2V]^G$ requires a generating set of cardinality $12$ while Example \ref{exam103} shows that $\F_2[V\oplus V^*]^G$ can be generated by $8$ elements. 
\hbo}\end{rem}

We close this section by sketching the proof of Theorem \ref{thm1}. First of all, we find all relations among the elements of $\B$ in Section 2; Section 3 is devoted to proving Lemma \ref{lem3.1}, a key result that  describes all non-trivial $\tete$s over $\B$; Section 4 gives a detailed proof to Theorem \ref{thm1}, based on the following classical strategy for proving that a subalgebra $A$ of an invariant ring $k[W]^G$ is actually equal to $k[W]^G$ itself:
\begin{enumerate}
  \item $k[W]$ is integral over $A$;
  \item The invariant field $k(W)^G$ and the field of fractions of $A$ are equal;
  \item $A$ is normal.
\end{enumerate}
See, for example, \cite[Proposition 10.0.8]{CW11}.

\section{Relations}
\setcounter{equation}{0}
\renewcommand{\theequation}
{2.\arabic{equation}}
\setcounter{theorem}{0}
\renewcommand{\thetheorem}
{2.\arabic{theorem}}


\noindent This section is devoted to finding all the relations among the generating invariants of $\F_p[V\oplus V^*]^G$ in Theorem \ref{thm1}. The following relation (\ref{R0}) is key and initial, involving the invariants $u_0, u_1, y_n$ and $x_i, N(y_i)$ for all $i=1,2,\dots, n-1$.

\begin{lem} \label{lem201}
\begin{equation}
\label{R0}\tag{$R_0$}
u_0^p-u_{1}\cdot y_n^{p-1}-\sum_{i=1}^{n-1}x_i^p\cdot N(y_i)=0.
\end{equation}
\end{lem}
\begin{proof}
Combining (\ref{eq2.3}) and (\ref{eq2.5}), we have
\begin{eqnarray*}
u_0^p-u_{1}\cdot y_n^{p-1} &=& \left(\sum_{i=1}^nx_i^p\cdot y_i^p\right)- \left(\sum_{i=1}^nx_i^p\cdot y_i\cdot y_n^{p-1}\right)\\
&=&\sum_{i=1}^{n-1}x_i^p\cdot \left(y_i^p-y_i\cdot y_n^{p-1}\right)
\end{eqnarray*}
which rearranges to give $\sum\limits_{i=1}^{n-1}x_i^p\cdot N(y_i).$
\end{proof}

Suppose $F:\F_p[V\oplus V^*]\ra \F_p[V\oplus V^*]$ denotes the Frobenius homomorphism defined by
$x_i\mapsto x_i^p$ and $y_i\mapsto y_i$ for all $i=1,2, \ldots, n$. Then $u_j=F^{j}(u_0)$ for all $j\in\N^+$.
We may find further relations among these invariants $u_1, \dots, u_{n-2}, y_n$, and $x_i, N(y_i)$ for all $i=1,2,\dots, n-1$ via repeatedly applying the Frobenius homomorphism $F$.

\begin{coro} \label{coro202}
For $j=1,2,\dots,n-3$, we have
\begin{equation}
\label{Rj}\tag{$R_j$}
u_j^p-u_{j+1}\cdot y_n^{p-1}-\sum_{i=1}^{n-1}x_i^{p^{j+1}}\cdot N(y_i)=0.
\end{equation}
\end{coro}
\begin{proof}
Applying the Frobenius homomorphism $F$ to the relation (\ref{R0}) obtains $$F(u_0)^p-F(u_{1})\cdot F(y_n)^{p-1}-\sum_{i=1}^{n-1}F(x_i)^p\cdot F(N(y_i))=0,$$ which finalizes to
\begin{equation}
\label{R1}\tag{$R_1$}
u_1^p-u_{2}\cdot y_n^{p-1}-\sum_{i=1}^{n-1}x_i^{p^2}\cdot N(y_i)=0.
\end{equation}
Similarly, applying the Frobenius homomorphism $F$ to (\ref{R1}) can obtain the relation $(R_2)$, and in general, the relation
$(R_{j})$ can be obtained by applying $F$ to $(R_{j-1})$, for all $j=1,2,\dots,n-3$.
\end{proof}

To articulate the last relation among the invariants, we define
\begin{eqnarray*}
\Delta_1&:= & \sum_{i=1}^{n-2} (-1)^i\cdot d_{i,n-1}\cdot u_{n-2-i}^p+(-1)^{n-1}\cdot d_{n-1,n-1}\cdot  u_0\cdot y_n^{p-1} \\
\Delta_2&:=& \sum_{i=0}^{n-2} (-1)^i\cdot d_{i,n-1}\cdot \left(\sum_{j=1}^{n-1} x_j^{p^{n-1-i}}\cdot N(y_j)\right)
\end{eqnarray*}
where $d_{0,n-1}:=1$.

\begin{lem} \label{lem203}
We have the following relation:
\begin{equation}
\label{RS}\tag{$R_{n-2}$}
u_{n-2}^p-N(x_n)\cdot y_n^p+\Delta_1-\Delta_2=0.
\end{equation}
\end{lem}

\begin{proof} Note that
\begin{eqnarray*}
u_{n-2}^p-N(x_n)\cdot y_n^p&=&\left(\sum_{j=1}^nx_j^{p^{n-2}}\cdot y_j\right)^p-\left(x_n^{p^{n-1}}+\sum_{i=1}^{n-1} (-1)^i\cdot d_{i,n-1}\cdot x_n^{p^{n-1-i}}\right)\cdot y_n^p\\
&=&\sum_{j=1}^{n-1}x_j^{p^{n-1}}\cdot y_j^p-\sum_{i=1}^{n-1} (-1)^i\cdot d_{i,n-1}\cdot x_n^{p^{n-1-i}}\cdot y_n^p\\
&=&\sum_{j=1}^{n-1}x_j^{p^{n-1}}\cdot y_j^p-\sum_{i=1}^{n-2} (-1)^i\cdot d_{i,n-1}\cdot x_n^{p^{n-1-i}}\cdot y_n^p-(-1)^{n-1}\cdot d_{n-1,n-1}\cdot x_n\cdot y_n\cdot y_n^{p-1}\\
&=&\sum_{j=1}^{n-1}x_j^{p^{n-1}}\cdot y_j^p-\sum_{i=1}^{n-2} (-1)^i\cdot d_{i,n-1}\cdot \left(u_{n-2-i}^p-\sum_{j=1}^{n-1}x_j^{p^{n-1-i}}\cdot y_j^p\right)\\
&&-(-1)^{n-1}\cdot d_{n-1,n-1}\cdot \left(u_0-\sum_{j=1}^{n-1}x_j\cdot y_j\right)\cdot y_n^{p-1}\\%
&=&\sum_{i=0}^{n-2} (-1)^i\cdot d_{i,n-1}\cdot\sum_{j=1}^{n-1}x_j^{p^{n-1-i}}\cdot y_j^p-\sum_{i=1}^{n-2} (-1)^i\cdot d_{i,n-1}\cdot u_{n-2-i}^p\\
&&-(-1)^{n-1}\cdot d_{n-1,n-1}\cdot u_0\cdot y_n^{p-1}+(-1)^{n-1}\cdot d_{n-1,n-1}\cdot\sum_{j=1}^{n-1}x_j\cdot y_n^{p-1}\cdot y_j\\
&=&\sum_{i=0}^{n-2} (-1)^i\cdot d_{i,n-1}\cdot\sum_{j=1}^{n-1}x_j^{p^{n-1-i}}\cdot \left(N(y_j)+y_n^{p-1}\cdot y_j\right)-\sum_{i=1}^{n-2} (-1)^i\cdot d_{i,n-1}\cdot u_{n-2-i}^p\\
&&-(-1)^{n-1}\cdot d_{n-1,n-1}\cdot u_0\cdot y_n^{p-1}+(-1)^{n-1}\cdot d_{n-1,n-1}\cdot\sum_{j=1}^{n-1}x_j\cdot y_n^{p-1}\cdot y_j\\
&=&\Delta_2-\Delta_1+\sum_{i=0}^{n-1} (-1)^i\cdot d_{i,n-1}\cdot\sum_{j=1}^{n-1}x_j^{p^{n-1-i}}\cdot y_n^{p-1}\cdot y_j\\
&=&\Delta_2-\Delta_1+\left(\sum_{i=0}^{n-1} (-1)^i\cdot d_{i,n-1}\cdot\left(\sum_{j=1}^{n-1}x_j^{p^{n-1-i}}\cdot y_j\right)\right)\cdot y_n^{p-1}\\
&=&\Delta_2-\Delta_1.
\end{eqnarray*}
The last equation holds because
$$\sum_{i=0}^{n-1} (-1)^i\cdot d_{i,n-1}\cdot\left(\sum_{j=1}^{n-1}x_j^{p^{n-1-i}}\cdot y_j\right)=0$$
which have been proved by Chen-Wehlau \cite[Lemma 3]{CW17}  in variables $x_1,\dots,x_{n-1},y_1,\dots,y_{n-1}$.
Therefore, $u_{n-2}^p-N(x_n)\cdot y_n^p+\Delta_1-\Delta_2=0$, as desired.
\end{proof}

\section{Non-trivial T\^ete-\`a-T\^ete}
\setcounter{equation}{0}
\renewcommand{\theequation}
{3.\arabic{equation}}
\setcounter{theorem}{0}
\renewcommand{\thetheorem}
{3.\arabic{theorem}}

\noindent This section describes all non-trivial $\tete$s over $\B$. Let us take the grevlex ordering on $\F_p[V\oplus V^*]^G$ with $x_n>x_{n-1}>\cdots>x_1>y_n>y_{n-1}>\cdots>y_1$. Note that $\LT(N(x_n))=x_n^{p^{n-1}}, \LT(N(y_i))=-y_n^{p-1}\cdot y_i$, and $\LT(u_j)=x_n^{p^j}\cdot y_n$
for $1\leqslant i\leqslant n-1$ and $0\leqslant j\leqslant n-2$.

\begin{lem}\label{lem3.1}
All non-trivial $\tete$s over $\B$ are given by:
$$u_j^p-u_{j+1}\cdot y_n^{p-1}, 0\leqslant j\leqslant n-3 \text{ and } u_{n-2}^p-N(x_n)\cdot y_n^p.$$
\end{lem}

\begin{proof}
Finding all non-trivial $\tete$s over $\B$ is equivalent to finding all relations among the leading terms of elements of $\B$.
We first note that these relations only occur involving the leading terms of $N(x_n), y_n$, and $u_0,\dots,u_{n-2}$.

Let $S:=\F_p[X, Y, U_0, \ldots, U_{n-2}]$  be a polynomial algebra with variables $X, Y, U_0, \ldots, U_{n-2}$ and
$B$ denote the subalgebra of $\F_p[x_n,y_n]$ generated by the leading terms of $N(x_n), y_n$, and $u_0,\dots,u_{n-2}$.
Consider the standard homomorphism $\varphi:S\ra B$
defined by
$$X\mapsto \LT(N(x_n)), Y\mapsto\LT(y_n), U_j\mapsto\LT(u_j)$$
for all $j=0,1,\dots,n-2.$ We observe that
$$T_j:=U_j^p-U_{j+1}\cdot Y^{p-1}\in \ker(\varphi)$$ for all $0\leqslant j\leqslant n-3$, and furthermore,
$$T_{n-2}:=U_{n-2}^p-X\cdot Y^p\in \ker(\varphi).$$
In fact, $\varphi(U_j^p-U_{j+1}\cdot Y^{p-1}) = \varphi(U_j)^p-\varphi(U_{j+1})\cdot\varphi(Y)^{p-1} = \LT(u_j)^p- \LT(u_{j+1})\cdot\LT(y_n)^{p-1} =(x_n^{p^j}\cdot y_n)^p-x_n^{p^{j+1}}\cdot y_n\cdot y_n^{p-1}=0$, and
$\varphi(U_{n-2}^p-X\cdot Y^p)=\LT(u_{n-2})^p-\LT(N(x_n))\cdot \LT(y_n)^p=(x_n^{p^{n-2}}\cdot y_n)^p-x_n^{p^{n-1}}\cdot y_n^p=0$.

We \textbf{claim} that $\ker(\varphi)$ is generated by $\{T_0,T_1,\dots,T_{n-2}\}$.
To prove this claim, we consider the quotient algebra $\ol{S}:=S/(T_0,T_1,\dots,T_{n-2})$
and the standard surjective homomorphism
$$\ol{\varphi}:\ol{S}\ra B.$$
The key fact we observe is that $\ol{S}$ is an integral domain. In fact,  we first note that
$$ T_i\equiv U_i^p \mod (Y)$$
for $i=0,1,\dots,n-2$. As $S$ is an integral domain, $X, Y, U_0,U_1,\dots,U_{n-2}$ is a regular sequence in $S$. Thus
$X, Y, U_0^p,U_1^p,\dots,U_{n-2}^p$ is a regular sequence in $S$. Therefore, $X, Y, T_0,T_1,\dots,T_{n-2}$ is  also a regular sequence in $S$. This means that $Y$ is not a zerodivisor in $\ol{S}$ and $\ol{S}$ can be embedded into $\ol{S}[Y^{-1}]$. Moreover, in $\ol{S}[Y^{-1}]$, the relations $T_0,T_1,\dots,T_{n-2}$ imply that $U_1,U_2,\dots,U_{n-2},X$ can be expressed in terms of $U_0$ and $Y^{-1}$. This implies $\ol{S}[Y^{-1}]\cong \F_p[Y,U_0][Y^{-1}]$, which is an integral domain. Therefore, the quotient ring $\ol{S}$, as a subring of $\ol{S}[Y^{-1}]$, is an integral domain.

Now we need to show that the map $\ol{\varphi}$ is an isomorphism. As $\ol{S}$ is integral domain and
$B\cong \ol{S}/\ker(\ol{\varphi})$,  the height of $\ker(\ol{\varphi})$ is equal to $\dim(\ol{S}) - \dim(B) = 2-2= 0$. On the other hand, note that $B\subseteq \F_p[x_n, y_n]$ is an integral domain, thus $\ker(\ol{\varphi})$ is a prime ideal of $\ol{S}$. Since $\{0\} \subseteq \ker(\ol{\varphi})$ and $\{0\}$ is a prime ideal in $\ol{S}$, we see that
$\ker(\ol{\varphi}) = 0$, i.e., $\ol{\varphi}$ is injective. Therefore, $\ol{\varphi}$ is an isomorphism and the claim holds.

This also shows that all non-trivial $\tete$s over $\B$ are given by $u_j^p-u_{j+1}\cdot y_n^{p-1}$ and $u_{n-2}^p-N(x_n)\cdot y_n^p$, where $0\leqslant j\leqslant n-3$.
\end{proof}

\section{Proof of Theorem \ref{thm1}}
\setcounter{equation}{0}
\renewcommand{\theequation}
{4.\arabic{equation}}
\setcounter{theorem}{0}
\renewcommand{\thetheorem}
{4.\arabic{theorem}}

\noindent To complete the proof of Theorem \ref{thm1}, we let $A$ be the $\F_p$-subalgebra of $\F_p[V\oplus V^*]$ generated by the set $\B$. Clearly, $A\subseteq \F_p[V\oplus V^*]^G$. Our argument will be separated into the following five steps.

\textbf{Step 1}. We \textit{claim} that $\B$ is a SAGBI basis for $A$. To prove this claim, it suffices to show that every non-trivial \tete over $\B$ subducts to zero.
We have shown in Lemma \ref{lem3.1} that $u_j^p-u_{j+1}\cdot y_n^{p-1} (0\leqslant j\leqslant n-3) \text{ and } u_{n-2}^p-N(x_n)\cdot y_n^p$ give all the non-trivial $\tete$s over $\B$. Thus, we only need to show that these $\tete$ differences
$$u_j^p-u_{j+1}\cdot y_n^{p-1} (0\leqslant j\leqslant n-3) \text{ and } u_{n-2}^p-N(x_n)\cdot y_n^p$$
subduct to zero. In fact, the relations (\ref{Rj}) $(0\leqslant j\leqslant n-3)$ in Lemma \ref{lem201}
and Corollary \ref{coro202} imply that each $u_j^p-u_{j+1}\cdot y_n^{p-1}$ subducts to zero.
The relation (\ref{RS}) in Lemma \ref{lem203} shows that $u_{n-2}^p-N(x_n)\cdot y_n^p$ also subducts to zero.
Therefore, $\B$ is a SAGBI basis for $A$.

\textbf{Step 2}. We \textit{claim} that $A$ and $\F_p[V\oplus V^*]^G$ have the same field of fractions.
Since $A$ is a subalgebra of $\F_p[V\oplus V^*]^G$ and the field  of fractions of  $\F_p[V\oplus V^*]^G$ is
equal to the invariant field $\F_p(V\oplus V^*)^G$, it suffices to show that $\F_p(V\oplus V^*)^G$ is contained in $\Quot(A)$, the field of fractions of $A$. As $A$ has Krull dimension $2n$ and the relations (\ref{R0}), (\ref{Rj}) for all $j=1,2,\dots, n-3$ and (\ref{RS}) imply that
$$\Quot(A)=\F_p(x_1,\dots,x_{n-1}, y_n,N(y_{n-1}),\dots, N(y_1),u_0).$$
On the other hand, \cite[Proposition 1.1]{BK11} shows that $\F_p(V\oplus V^*)^G$, as a field extension over $\F_p$, is generated by $\F_p[V]^G$, $\F_p[V^*]^G$, and $u_0$. Note that $\F_p[V]^G=\F_p[x_1,\dots,x_{n-1},N(x_n)]$ and $\F_p[V^*]^G=\F_p[y_n,N(y_{n-1}),\dots, N(y_1)]$. Thus
$$\F_p(V\oplus V^*)^G=\F_p\left(x_1,\dots,x_{n-1}, N(x_n),y_n,N(y_{n-1}),\dots, N(y_1), u_0\right)\subseteq\Quot(A).$$
Here, note that $N(x_n)$ belongs to $\Quot(A)$ by the relation (\ref{RS}).

\textbf{Step 3}. We observe that $A$ contains an hsop $\{x_1,\dots,x_{n-1},N(x_n), y_n,N(y_{n-1}),\dots, N(y_1)\}$ for $\F_p[V\oplus V^*]$, i.e., $\F_p[V\oplus V^*]$ is an integral extension over $A$.

\textbf{Step 4}. We \textit{claim} that $A$ is a complete intersection. Let $P: =\F_p[X_1,\dots,X_n,Y_n,\dots,Y_1,U_0,\dots, U_{n-2}]$. Consider the standard homomorphism $\rho$ of $\F_p$-algebras $\rho: P\ra A$ defined by sending
$$X_j,X_n,Y_n,Y_j,U_0,\dots, U_{n-2} \textrm{ to } x_j,N(x_n),y_n,\\ N(y_j),u_0,\dots, u_{n-2}$$ for all $j=1,2,\dots,n-1$, respectively. Note that the following polynomials
\begin{eqnarray*}
\R_0&:=&U_0^p-U_{1}\cdot Y_n^{p-1}-\sum_{i=1}^{n-1}X_i^p\cdot Y_i;\\
\R_1&:=&U_1^p-U_{2}\cdot Y_n^{p-1}-\sum_{i=1}^{n-1}X_i^{p^{2}}\cdot Y_i;\\
\vdots\\
\R_{n-3}&:=&U_{n-3}^p-U_{n-2}\cdot Y_n^{p-1}-\sum_{i=1}^{n-1}X_i^{p^{n-2}}\cdot Y_i;\\
\R_{n-2}&:=&U_{n-2}^p-X_n\cdot Y_n^p+\sum_{i=1}^{n-2} (-1)^i\cdot d_{i,n-1}\cdot U_{n-2-i}^p+(-1)^{n-1}\cdot d_{n-1,n-1}\cdot U_0\cdot Y_n^{p-1}\\
&&-\sum_{i=0}^{n-2} (-1)^i\cdot d_{i,n-1}\cdot \left(\sum_{j=1}^{n-1} X_j^{p^{n-1-i}}\cdot Y_j\right)
\end{eqnarray*}
are the homogeneous elements of $\ker(\rho)$.

Let $\ol{P}:=P/(\R_{0},\dots,\R_{n-2})$. Then there exists a standard surjection $$\ol{\rho}: \ol{P}\ra A$$ induced from $\rho$.
First of all, we \textit{claim} that $\ol{\rho}$ is an isomorphism and thus $A$ is a complete intersection.

Note that for all $j=0,\dots, n-2$, we have
$$\R_{j}\equiv U_{j}^p \mod (X_1,\dots,X_n,Y_n,\dots,Y_1).$$
Since $P$ is a polynomial algebra, $X_1,\dots,X_n,Y_n,\dots,Y_1,U_0,\dots, U_{n-2}$ is a regular sequence in $P$, as is $X_1,\dots,X_n,Y_n,\dots,Y_1,U_0^p,\dots, U_{n-2}^p$. Thus $$X_1,\dots,X_n,Y_n,\dots,Y_1,\R_0,\dots, \R_{n-2}$$ is also a regular sequence in $P$. Hence, $Y_n$ is not a zerodivisor in $\ol{P}$ and $\ol{P}$ can be embedded into $\ol{P}\left[\frac{1}{Y_n}\right]$. Working in $\ol{P}$, the images of $\R_{0},\dots,\R_{n-2}$ are zero.  This means that the images of $X_n, U_1,\dots,U_{n-2}$ can be expressed by the images of $X_1,\dots,X_{n-1},Y_n,\dots,Y_1,U_0$ and $\frac{1}{Y_n}$. Thus,
$$\ol{P}\left[\frac{1}{Y_n}\right]\cong \F_p[X_1,\dots,X_{n-1},Y_n,\dots,Y_1,U_0]\left[\frac{1}{Y_n}\right].$$
The latter is a localization of a polynomial ring and so it is an integral domain. Therefore, the ring $\ol{P}$, as a subring of $\ol{P}\left[\frac{1}{Y_n}\right]$, is an integral domain as well. Moreover, $\ker(\ol{\rho})$ is a prime ideal in $\ol{P}$. Thus the height of $\ker(\ol{\rho})$ equals to $\dim(\ol{P})-\dim(A)=0$; see, for example, \cite[Theorem 1.8A]{Har77}. Note that $\{0\}\subseteq \ker(\ol{\rho})$ and $\{0\}$ is a prime ideal in $\ol{P}$, we have $\ker(\ol{\rho})=\{0\}$, i.e., $\ol{\rho}$ is injective. Therefore, $\ol{\rho}$ is an isomorphism and $A$ is a complete intersection.

\textbf{Step 5}.  Finally, we \textit{claim} that $A$ is normal. Recall that every factorial domain (i.e., UFD) is normal, by Nagata's Lemma (\cite[Lemma 19.20]{Eis95}), we would like to show that a localization of $A$ at a prime element is factorial.
By \cite[Proposition 3.1]{CT19}, we see that
$$A\left[\frac{1}{y_n}\right]=\F_p[V\oplus V^*]^G\left[\frac{1}{y_n}\right]$$ is a normal domain. Thus, to show that $A$ is normal, it is sufficient to show that $y_n$ is a prime element of $A$. By Step 4, it suffices to show that $Y_n$ is a prime element of $\ol{P}$.
Note that
\begin{eqnarray*}
A/ (y_n)\cong \ol{P}/(Y_n)\cong \F_p[X_1,\dots,X_n,Y_{n-1},\dots, Y_1,U_0,\dots,U_{n-2}]/(\wt{\R}_0,\wt{\R}_1,\dots,\wt{\R}_{n-2}),
\end{eqnarray*}
where $\wt{\R}_j$ denote the images  of $\R_j$ in $P/(Y_n)$ for all $j=0,1,\dots, n-2$.

Working in  $\ol{P}/(Y_n)$, the relations $\wt{\R}_j=0$ can be written as the following linear system:
$$
\scriptsize\begin{pmatrix}
X_1^p & X_2^p & \cdots & X_{n-1}^p \\
X_1^{p^2} & X_2^{p^2} & \cdots & X_{n-1}^{p^2} \\
\vdots & \vdots & \cdots & \vdots \\
\sum\limits_{i=0}^{n-2}(-1)^id_{i, n-1}X_1^{p^{n-1-i}} & \sum\limits_{i=0}^{n-2}(-1)^id_{i, n-1}X_{2}^{p^{n-1-i}} & \cdots & \sum\limits_{i=0}^{n-2}(-1)^id_{i,n-1}X_{n-1}^{p^{n-1-i}}
\end{pmatrix}\cdot\begin{pmatrix}
Y_1 \\
Y_2 \\
\vdots \\
Y_{n-1}
\end{pmatrix}=\begin{pmatrix}
U_0 \\
U_1^p \\
\vdots \\
U_{n-2}+\sum\limits_{i=1}^{n-2}(-1)^{i}d_{i,n-1}U_{n-2-i}^p
\end{pmatrix}.
$$
Let $\Delta$ denote the coefficient matrix of the above system. Applying elementary row transformations to $\Delta$ (in fact, multiplying the $i$-th row  with $(-1)^{n-i}d_{n-1-i, n-1}$ and adding them to the last row, where $i=1,2,\dots, n-2$), we obtain a new matrix $\Delta'$ which has the same determinant as $\Delta$. Note that $d_{0, n-1}=1$, thus
$$\det(\Delta)=\det(\Delta')=\det\begin{pmatrix}
X_1 & X_2 & \cdots & X_{n-1} \\
X_1^{p} & X_2^{p} & \cdots & X_{n-1}^{p} \\
\vdots & \vdots & \cdots & \vdots \\
X_1^{p^{n-2}} & X_{2}^{p^{n-2}} & \cdots & X_{n-1}^{p^{n-2}}
\end{pmatrix}^p.$$
Hence, $\R_0,\dots, \R_{n-2}, \det(\Delta)$ is a regular sequence in $P$. This means that $\det(\Delta)$ is a nonzero-divisor in $\ol{P}$ and $\ol{P}/(Y_n)$ can be embedded into $\ol{P}[\det(\Delta)^{-1}]/(Y_n)$, and it suffices to show that the latter is an integral domain. Let $\Delta_j$ denote the matrix obtained from $\Delta$ by replacing the $j$-th column with
$(U_0, -U_1^p,\dots, U_{n-2}+\sum\limits_{i=1}^{n-2}(-1)^{i}\cdot d_{i,n-1}\cdot U_{n-2-i}^p)^T$. By the Cramer's rule, we may express $Y_1,Y_2,\dots, Y_{n-1}$ in terms of $X_1,\dots, X_n, U_0, \dots, U_{n-2}$ and $\det(\Delta)^{-1}$, that is, $$Y_j=\frac{\det(\Delta_j)}{\det(\Delta)}.$$ Hence, $\ol{P}[\det(\Delta)^{-1}]/(Y_n)\cong \F_p[X_1,\dots, X_n, U_0, \dots, U_{n-2}][\det(\Delta)^{-1}]$, which is an integral domain. Therefore, $\ol{P}/(Y_n)$, as a subring of $\ol{P}[\det(\Delta)^{-1}]/(Y_n)$, is also an integral domain. This proves that $A$ is normal.

\section*{Acknowledgments}
\noindent The author would like to thank her Ph.D advisor Professor Yin Chen for his supervision and help. She also wants to thank Professors H. Eddy A. Campbell, Xianhui Fu, and David L. Wehlau for their careful reading, comments, and encouragements.  Many thanks go to the anonymous referee for his/her valuable comments and suggestions on the first version of this paper. 
The symbolic computation language MAGMA \cite{BCP97} (http://magma.maths.usyd.edu.au/) was very helpful.


\raggedright

\begin{bibdiv}
  \begin{biblist}

\bib{BCP97}{article}{
   author={Bosma, Wieb},
   author={Cannon, John},
   author={Playoust, Catherine},
   title={The Magma algebra system. I. The user language},
   journal={J. Symbolic Comput.},
   volume={24},
   date={1997},
   number={3-4},
   pages={235--265},
   issn={0747-7171},
}



\bib{BK11}{article}{
   author={Bonnaf\'e, C\'edric},
   author={Kemper, Gregor},
   title={Some complete intersection symplectic quotients in positive
   characteristic: invariants of a vector and a covector},
   journal={J. Algebra},
   volume={335},
   date={2011},
   pages={96--112},
   issn={0021-8693},
}




\bib{CSW10}{article}{
   author={Campbell, H. Eddy A.},
   author={Shank, R. James},
   author={Wehlau, David L.},
   title={Vector invariants for the two-dimensional modular representation
   of a cyclic group of prime order},
   journal={Adv. Math.},
   volume={225},
   date={2010},
   number={2},
   pages={1069--1094},
   issn={0001-8708},
}





\bib{CSW13}{article}{
   author={Campbell, H. Eddy A.},
   author={Shank, R. James},
   author={Wehlau, David L.},
   title={Rings of invariants for modular representations of elementary
   abelian $p$-groups},
   journal={Transform. Groups},
   volume={18},
   date={2013},
   number={1},
   pages={1--22},
   issn={1083-4362},
}



\bib{CW11}{book}{
   author={Campbell, H. Eddy A.},
   author={Wehlau, David L.},
   title={Modular invariant theory},
   series={Encyclopaedia of Mathematical Sciences},
   volume={139},
   publisher={Springer-Verlag, Berlin},
   date={2011},
   pages={xiv+233},
   isbn={978-3-642-17403-2},
}

\bib{Che14}{article}{
   author={Chen, Yin},
   title={On modular invariants of a vector and a covector},
   journal={Manuscripta Math.},
   volume={144},
   date={2014},
   number={3-4},
   pages={341--348},
}

\bib{Che21}{article}{
   author={Chen, Yin},
   title={Relations between modular invariants of a vector and a covector in
   dimension two},
   journal={Canad. Math. Bull.},
   volume={64},
   date={2021},
   number={4},
   pages={820--827},
   issn={0008-4395},
}




\bib{CT19}{article}{
   author={Chen, Yin},
   author={Tang, Zhongming},
   title={Vector invariant fields of finite classical groups},
   journal={J. Algebra},
   volume={534},
   date={2019},
   pages={129--144},
   issn={0021-8693},
}


\bib{CW17}{article}{
   author={Chen, Yin},
   author={Wehlau, David L.},
   title={Modular invariants of a vector and a covector: a proof of a
   conjecture of Bonnaf\'{e} and Kemper},
   journal={J. Algebra},
   volume={472},
   date={2017},
   pages={195--213},
   issn={0021-8693},
}


\bib{CW19}{article}{
   author={Chen, Yin},
   author={Wehlau, David L.},
   title={On invariant fields of vectors and covectors},
   journal={J. Pure Appl. Algebra},
   volume={223},
   date={2019},
   number={5},
   pages={2246--2257},
   issn={0022-4049},
}

\bib{Chu03}{article}{
   author={Chuai, Jianjun},
   title={Two-dimensional vector invariant rings of abelian $p$-groups},
   journal={J. Algebra},
   volume={266},
   date={2003},
   number={1},
   pages={362--373},
   issn={0021-8693},
}


\bib{DK15}{book}{
   author={Derksen, Harm},
   author={Kemper, Gregor},
   title={Computational invariant theory},
   series={Encyclopaedia of Mathematical Sciences},
   volume={130},
   edition={Second enlarged edition},
   publisher={Springer, Heidelberg},
   date={2015},
   pages={xxii+366},
   isbn={978-3-662-48420-3},
   isbn={978-3-662-48422-7},
}

\bib{Dic11}{article}{
   author={Dickson, Leonard E.},
   title={A fundamental system of invariants of the general modular linear
   group with a solution of the form problem},
   journal={Trans. Amer. Math. Soc.},
   volume={12},
   date={1911},
   number={1},
   pages={75--98},
}


\bib{Eis95}{book}{
   author={Eisenbud, David},
   title={Commutative algebra},
   series={Graduate Texts in Mathematics},
   volume={150},
   note={With a view toward algebraic geometry},
   publisher={Springer-Verlag, New York},
   date={1995},
   pages={xvi+785},
   isbn={0-387-94268-8},
   isbn={0-387-94269-6},
}

\bib{FS16}{article}{
   author={Fleischmann, Peter},
   author={Shank, James},
   title={The invariant theory of finite groups},
   conference={
      title={Algebra, logic and combinatorics},
   },
   book={
      series={LTCC Adv. Math. Ser.},
      volume={3},
      publisher={World Sci. Publ., Hackensack, NJ},
   },
   date={2016},
   pages={105--138},
}

\bib{Har77}{book}{
   author={Hartshorne, Robin},
   title={Algebraic geometry},
   series={Graduate Texts in Mathematics, No. 52},
   publisher={Springer-Verlag, New York-Heidelberg},
   date={1977},
   pages={xvi+496},
   isbn={0-387-90244-9},
}

\bib{Kem96}{article}{
   author={Kemper, Gregor},
   title={Calculating invariant rings of finite groups over arbitrary
   fields},
   journal={J. Symbolic Comput.},
   volume={21},
   date={1996},
   number={3},
   pages={351--366},
}

\bib{Weh13}{article}{
   author={Wehlau, David L.},
   title={Invariants for the modular cyclic group of prime order via
   classical invariant theory},
   journal={J. Eur. Math. Soc.},
   volume={15},
   date={2013},
   number={3},
   pages={775--803},
   issn={1435-9855},
}


  \end{biblist}
\end{bibdiv}
\end{document}